\documentclass[final,3p]{elsarticle}
\usepackage[latin1]{inputenc}
 \usepackage{graphics}
 \usepackage{graphicx}
 \usepackage{epsfig}
\usepackage{amssymb}
 \usepackage{amsthm}
 \usepackage{lineno}
 \usepackage{amsmath}
   \numberwithin{equation}{section}
\usepackage{mathrsfs}
\usepackage{color}

\newtheorem{thm}{Theorem}[section]

\newtheorem{lem}[thm]{Lemma}
\newtheorem{prop}[thm]{Proposition}
\newtheorem{defn}[thm]{Definition}

\biboptions{sort&compress,square}
\allowdisplaybreaks
\begin{document}
\begin{frontmatter}
\author{Tong Wu}
\ead{wut977@nenu.edu.cn}
\author{Yong Wang\corref{cor}}
\ead{wangy581@nenu.edu.cn}
\cortext[cor]{Corresponding author.}

\address{School of Mathematics and Statistics, Northeast Normal University,
Changchun, 130024, China}

\title{ Super warped products with a semi-symmetric non-metric connection }
\begin{abstract}
In this paper, we define a semi-symmetric non-metric
 connection on super Riemannian manifolds. And we compute the curvature tensor and the Ricci tensor of a semi-symmetric non-metric
 connection on super warped product spaces. Next, we introduce two kinds of super warped product spaces with a semi-symmetric non-metric
 connection and give the conditions that two super warped product spaces with a semi-symmetric non-metric
 connection are the Einstein super spaces with a semi-symmetric non-metric
 connection.
\end{abstract}
\begin{keyword} Semi-symmetric non-metric connection; the curvature tensor; Ricci tensor; super warped product spaces; the Einstein super spaces.\\

\end{keyword}
\end{frontmatter}
\section{Introduction}
 \indent The (singly) warped product $B\times_hF$ of two pseudo-Riemannian manifolds $(B,g_B)$ and
    $(F,g_F)$ with a smooth function $h:B\rightarrow (0,\infty)$ is the product
    manifold $B\times F$ with the metric tensor $g=g_B\oplus
    h^2g_F.$ Here, $(B,g_B)$ is called the base manifold,
    $(F,g_F)$ is called as the fiber manifold and $h$ is called as
    the warping function. Generalized Robertson-Walker space-times
    and standard static space-times are two well-known warped
    product spaces. The concept of warped products was first introduced by
    Bishop and ONeil (see \cite{BO}) to construct examples of Riemannian
    manifolds with negative curvature. In Riemannian geometry,
    warped product manifolds and their generic forms have been used
    to construct new examples with interesting curvature properties
    since then. In \cite{DD}, F. Dobarro and E. Dozo had studied from the viewpoint of partial differential equations and variational methods,
    the problem of showing when a Riemannian metric of constant scalar curvature can be produced on a product manifolds by a warped product
    construction.
    In \cite{EJK}, Ehrlich, Jung and Kim got explicit solutions to warping function to have a constant scalar curvature for generalized
    Robertson-Walker space-times.
    In \cite{ARS}, explicit solutions were also obtained for the warping
    function to make the space-time as Einstein when the fiber is
    also Einstein.\\
      \indent  N. S. Agashe and M. R. Chafle introduced the notion of a semi-symmetric non-metric connection and studied some of its properties and submanifolds of a
Riemannian manifold with a semi-symmetric non-metric connection \cite{AC1, AC2}. In \cite{SO}, Sular and \"{O}zgur studied warped product manifolds with a
         semi-symmetric non-metric connection, they computed curvature of semi-symmetric non-metric connection
          and considered Einstein warped product manifolds with a semi-symmetric non-metric connection. In \cite{W1}, Wang studied the Einstein multiply warped products with a semi-symmetric metric connection and the multiply warped products with a semi-symmetric metric connection with constant scalar curvature.\\
  \indent On the other hand, in \cite{BG}, the definition of super warped product spaces was given. In \cite{GDMVR}, several new super warped product spaces were given and the authors also studied the Einstein
  equations with cosmological constant in these new super warped product spaces. In \cite{WY}, Wang studied super warped product spaces with a semi-symmetric metric connection. Our motivation is to study super warped product spaces with a semi-symmetric non-metric connection.\\
\indent In Section \ref{Section:2}, we state some definitions of super manifolds and super Riemannian metrics. We also define a semi-symmetric non-metric
 connection on super Riemannian manifolds and prove that there is a unique semi-symmetric non-metric
 connection on super Riemannian manifolds which is non-metric and has the semi-symmetric torsion. In Section \ref{Section:3}, we compute the curvature tensor and the Ricci tensor of a semi-symmetric non-metric
 connection on super warped product spaces. In Section \ref{Section:4}, we introduce two kinds of super warped product spaces with a semi-symmetric non-metric
 connection and give the conditions that two super warped product spaces with a semi-symmetric non-metric
 connection are the Einstein super spaces with a semi-symmetric non-metric
 connection.
\section{A semi-symmetric non-metric
 connection on super Riemannian manifolds}
\label{Section:2}
In this section, we give some definitions about Riemannian supergeometry.
\begin{defn}\label{def1} A locally $\mathbb{Z}_2$-ringed space is a pair $S:= (|S|, \mathcal{O}_S)$ where $|S|$ is a second-countable
Hausdorff space, and a $\mathcal{O}_S$ is a sheaf of $\mathbb{Z}_2$-graded $\mathbb{Z}_2$-commutative associative unital $\mathbb{R}$-algebras, such that the
stalks $\mathcal{O}_{S,p}$, $p\in |S|$ are local rings.
\end{defn}
 \indent In this context, $\mathbb{Z}_2$-commutative means that any two sections $s,t\in \mathcal{O}_S(|U|),~~|U|\subset|S|$ open,
 of homogeneous
degree $|s|\in \mathbb{Z}_2$
and $|t|\in \mathbb{Z}_2$
commute up to the sign rule
$st=(-1)^{|s||t|}ts$.
 $\mathbb{Z}_2$-ring
space $U^{m|n}:= (U,C^{\infty}_{U^m}\otimes \wedge \mathbb{R}^n)$, is called standard
superdomain where $C^{\infty}_{U^m}$ is the sheaf of smooth functions on $U$ and $\wedge\mathbb{R}^n$ is
the exterior algebra of $\mathbb{R}^n$. We can employ (natural) coordinates $x^I:=(x^a,\xi^A)$ on any $\mathbb{Z}_2$-domain, where $x^a$ form a coordinate system on $U$ and the $\xi^A$
are formal coordinates.
\begin{defn}\label{def2}
 A supermanifold of dimension $m|n$ is a super ringed space
$M=(|M|, \mathcal{O}_M )$ that is locally isomorphic to $\mathbb{R}^{m|n}$ and $|M|$ is a second countable
and Hausdorff topological space.
\end{defn}
The tangent sheaf $\mathcal{T}M$ of a $\mathbb{Z}_2$-manifold $M$ is defined as the sheaf of derivations of sections of the structure
sheaf, i.e., $\mathcal{T}M(|U|) := {\rm Der}(\mathcal{O}_M(|U|)),$ for arbitrary open set $|U|\subset |M|.$ Naturally, this is a sheaf of locally free $\mathcal{O}_M$-modules. Global sections of the tangent sheaf are referred to as {\it vector fields}. We denote the $\mathcal{O}_M (|M|)$-module
of vector fields as ${\rm Vect}(M)$. The dual of the tangent sheaf is the {\it cotangent sheaf}, which we denote as $\mathcal{T}^*M$.
This is also a sheaf of locally free $\mathcal{O}_M$-modules. Global section of the cotangent sheaf we will refer to as {\it one-forms}
and we denote the $\mathcal{O}_M(|M|)$-module of one-forms as $\Omega^1(M)$.
\begin{defn}\label{def3}
 A Riemannian metric on a $\mathbb{Z}_2$-manifold M is a $\mathbb{Z}_2$-homogeneous, $\mathbb{Z}_2$-symmetric, non-degenerate,
$\mathcal{O}_M$-linear morphisms of sheaves $\left<-,-\right>_g:~~\mathcal{T}M\otimes \mathcal{T}M\rightarrow \mathcal{O}_M.$
A $\mathbb{Z}_2$-manifold equipped with a Riemannian metric is referred to as a Riemannian $\mathbb{Z}_2$-manifold.
\end{defn}
We will insist that the Riemannian metric is homogeneous with respect to the $\mathbb{Z}_2$-degree, and we will denote
the degree of the metric as $|g| \in \mathbb{Z}_2$.
Explicitly, a Riemannian metric has the following properties:\\
(1)$ |\left<X,Y\right>_g |= |X| + |Y |+ |g|,$\\
(2)$\left<X,Y\right>_g =(-1)^{|X||Y|}\left<Y,X\right>_g,$\\
(3) If $\left<X,Y\right>_g = 0$ for all $Y \in Vect(M),$ then $X = 0,$\\
(4) $\left<fX+Y,Z\right>_g =f\left<X,Z\right>_g +\left<Y,Z\right>_g ,$\\
for arbitrary (homogeneous) $ X, Y, Z \in {\rm Vect}(M)$ and $f \in C^{\infty}(M)$. We will say that a Riemannian metric is
even if and only if it has degree zero. Similarly, we will say that a Riemannian metric is odd if and only
if it has degree one. Any Riemannian metric we consider will be either even or odd as we will only be
considering homogeneous metrics.\\
\indent Now we recall the definition of the warped product of Riemannian $\mathbb{Z}_2$-manifolds. For details, see the section 2.3 in \cite{BG}. Let $M_1\times M_2$
be the product of two $\mathbb{Z}_2$-manifolds $M_1$ and $M_2$. Let $(M_i,g_i)( i=1,2)$ be Riemannian $\mathbb{Z}_2$-manifolds whose Riemannian metric are of the same $\mathbb{Z}_2$-degree. Let $\mu\in C^{\infty}
(M_1)$ be a degree $0$ invertible global functions that is strictly positive, i.e. $\varepsilon_{M_1}(\mu)$ a strictly positive function on $|M_1|$ where $\varepsilon$ is simply "throwing away" the formal coordinates. Then the warped product is defined as
$$M_1\times_\mu M_2:=(M_1\times M_2,g:=\pi_1^*g_1+(\pi^*_1\mu)\pi^*_2g_2),$$
where $\pi_i: M_1\times M_2\rightarrow M_i~~(i=1,2)$ is the projection. By Proposition 4 in \cite{BG}, the warped product $M_1\times_\mu M_2$ is a Riemannian $\mathbb{Z}_2$-manifold.\\
\begin{defn}\label{fff4}(Definition 9 in \cite{BG}) An affine connection on a $\mathbb{Z}_2$-manifold is a $\mathbb{Z}_2$-degree preserving map\\
$$\nabla:~~ {\rm Vect}(M)\times {\rm Vect}(M)\rightarrow {\rm Vect}(M);~~(X,Y)\mapsto \nabla_XY,$$
which satisfies the following\\
1) Bi-linearity $$\nabla_X(Y+Z)=\nabla_XY+\nabla_XZ;~~\nabla_{X+Y}Z=\nabla_XZ+\nabla_YZ,$$
2)$C^{\infty}(M)$-linearrity in the first argument
$$\nabla_{fX}Y=f\nabla_XY,$$
3)The Leibniz rule
$$\nabla_X(fY)=X(f)Y+(-1)^{|X||f|}f\nabla_XY,$$
for all homogeneous $X,Y,Z\in {\rm Vect}(M)$ and $f\in C^{\infty}(M)$.
\end{defn}
\begin{defn}\label{def4}(Definition 10 in \cite{BG})
 The torsion tensor of an affine connection \\
 $T_\nabla:~~{\rm Vect}(M)\otimes_{C^{\infty}(M)} {\rm Vect}(M)\rightarrow {\rm Vect}(M)$
 is
defined as
$$T_\nabla(X,Y):=\nabla_XY-(-1)^{|X||Y|}\nabla_YX-[X,Y],$$
for any (homogeneous) $X, Y \in {\rm Vect}(M)$. An affine connection is said to be symmetric if the torsion vanishes.
\end{defn}
\begin{defn}\label{def5}(Definition 11 in \cite{BG})
An affine connection on a Riemannian $\mathbb{Z}_2$-manifold $(M, g)$ is said to be metric compatible if
and only if
$$ X\left<Y,Z\right>_g=\left<\nabla_XY,Z\right>_g+(-1)^{|X||Y|}\left<Y,\nabla_XZ\right>_g,$$
for any $X, Y, Z\in  {\rm Vect}(M)$.
\end{defn}
\begin{thm}\label{thm1}(Theorem 1 in \cite{BG})There is a unique symmetric (torsionless) and metric compatible
affine connection $\nabla^L$ on a Riemannian $\mathbb{Z}_2$-manifold $(M, g)$ which satisfies the Koszul formula
\begin{align}\label{a1}
2\left<\nabla^L_XY,Z\right>_g&=X\left<Y,Z\right>_g+\left<[X,Y],Z\right>_g\nonumber\\
&+(-1)^{|X|(|Y|+|Z|)}(Y\left<Z,X\right>_g-\left<[Y,Z],X\right>_g)\nonumber\\
&-(-1)^{|Z|(|X|+|Y|)}(Z\left<X,Y\right>_g-\left<[Z,X],Y\right>_g),\nonumber\\
\end{align}
for all homogeneous $X,Y,Z\in {\rm Vect}(M)$.
\end{thm}
\begin{defn}\label{def6}(Definition 13 in \cite{BG})
The Riemannian curvature tensor of an affine connection
$$R_\nabla:~~{\rm Vect}(M)\otimes_{C^{\infty}(M)} {\rm Vect}(M)\otimes_{C^{\infty}(M)} {\rm Vect}(M)\rightarrow {\rm Vect}(M)$$
is defined as
\begin{equation}\label{a2}
R_\nabla(X, Y )Z =\nabla_X\nabla_Y-(-1)^{|X||Y|}\nabla_Y\nabla_X-\nabla_{[X,Y]}Z,\nonumber\\
\end{equation}
for all $X, Y$ and $Z \in {\rm Vect}(M)$.
\end{defn}
Directly from the definition it is clear that
\begin{equation}\label{a3}
R_\nabla(X, Y )Z =-(-1)^{|X||Y|}R_\nabla(Y,X)Z,
\end{equation}
for all $X, Y$ and $Z \in {\rm Vect}(M)$.
\begin{defn}\label{def7}(Definition 14 in \cite{BG})
 The Ricci curvature tensor of an affine connection is the symmetric rank-$2$ covariant tensor
defined as
 \begin{equation}\label{a4}
Ric_\nabla(X, Y ):=(-1)^{|\partial_{x^I}|(|\partial_{x^I}|+|X|+|Y|)}\frac{1}{2}\left[R_\nabla(\partial_{x^I},X)Y+(-1)^{|X||Y|}R_\nabla(\partial_{x^I},Y)X\right]^I,
\end{equation}
where $X,Y\in {\rm Vect}(M)$ and $[~~]^I$ denotes the coefficient of $\partial_{x^I}$ and $\partial_{x^I}$ is the natural frame of $\mathcal{T}M$.
\end{defn}
\begin{defn}\label{def8}(Definition 16 in \cite{BG})
Let $f \in C^{\infty}(M)$ be an arbitrary function on a Riemannian $\mathbb{Z}_2$-manifold $(M, g)$. The gradient
of $f$ is the unique vector field ${\rm grad}_gf$ such that
 \begin{equation}\label{a5}
X(f)=(-1)^{|f||g|}\left<X,{\rm grad}_gf\right>_g,
\end{equation}
for all $X \in {\rm Vect}(M)$.
\end{defn}
\begin{defn}\label{def9}(Definition 17 in \cite{BG})
Let $(M, g)$ be a Riemannian $\mathbb{Z}_2$-manifold
and let $\nabla^L$ be the associated Levi-Civita connection.
The covariant divergence is the map ${\rm Div}_L:{\rm Vect}(M)\rightarrow C^{\infty}(M)$, given by
\begin{equation}\label{a6}
{\rm Div}_L(X)=(-1)^{|\partial_{x^I}|(|\partial_{x^I}|+|X|)}(\nabla_{\partial_{x^I}}X)^I,
\end{equation}
for any arbitrary $X \in {\rm Vect}(M)$.
\end{defn}
\begin{defn}\label{def10}(Definition 18 in \cite{BG}) Let $(M, g)$ be a Riemannian $\mathbb{Z}_2$-manifold
and let $\nabla^L$ be the associated Levi-Civita connection.
The connection Laplacian (acting on functions) is the differential operator of  $\mathbb{Z}_2$-degree $|g|$ defined as
\begin{equation}\label{a7}
\triangle_g(f)={\rm Div}_L({\rm grad}_gf),
\end{equation}
for any and all $f \in C^{\infty}(M).$
\end{defn}
\begin{defn}\label{def11}Let $(M, g)$ be a Riemannian $\mathbb{Z}_2$-manifold and $P\in {\rm Vect}(M)$ which satisfied $|g|+|P|=0$
and we define a semi-symmetric non-metric connection $\widehat{\nabla}$ on $(M, g)$
\begin{equation}\label{a8}
{\widehat{\nabla}}_XY=\nabla^L_XY+X\cdot g(Y,P)=\nabla^L_XY+(-1)^{|X||Y|}g(Y,P)X,
\end{equation}
for any homogenous $X,Y\in {\rm Vect}(M)$ and where
$X\cdot f=(-1)^{|X||f|}fX$ for $f\in C^{\infty}(M)$.
\end{defn}
Obviously, we have $\widehat{\nabla}_{X+Y}Z=\widehat{\nabla}_{X}Z+\widehat{\nabla}_{Y}Z;~~\widehat{\nabla}_{X}(Y+Z)=\widehat{\nabla}_XY+\widehat{\nabla}_XZ,$ for any homogenous $X, Y, Z\in {\rm Vect}(M)$.
We can verify that $\widehat{\nabla}_XY$ satisfies the Definition \ref{fff4}, then $\widehat{\nabla}_XY$ is an affine connection. By Definition \ref{def4}, we get
\begin{equation}\label{a9}
T_{\widehat{\nabla}}(X,Y)=X\cdot g(Y,P)-(-1)^{|X||Y|}Y\cdot g(X,P).\\
\end{equation}
Then, we call that $\widehat{\nabla}_XY$ is a semi-symmetric connection. By Definition \ref{def5} and Definition \ref{def11}, we get
\begin{align}\label{a10}
&\left<\widehat{\nabla}_XY,Z\right>_g+(-1)^{|X||Y|}\left<Y,\widehat{\nabla}_XZ\right>_g\nonumber\\
&=X\left<Y,Z\right>_g+\left<X\cdot g(Y,P),Z\right>_g+(-1)^{|X||Y|}\left<Y,X\cdot g(Z,P)\right>_g\nonumber\\
&=X\left<Y,Z\right>_g+(-1)^{|Y||X|}g(Y,P)g(X,Z)+(-1)^{|X||Y|}(-1)^{|Z|(|X|+|Y|)}g(Z,P)g(Y,X).\nonumber\\
\end{align}
So $\widehat{\nabla}$ doesn't preserve the metric.

\begin{thm}\label{thm2}There is a unique non-metric compatible
affine connection $\widehat{\nabla}$ on a Riemannian $\mathbb{Z}_2$-manifold $(M, g)$ which satisfies (\ref{a9}) and (\ref{a10}).
\end{thm}
\begin{proof}By (\ref{a9}), we know that a semi-symmetric non-metric connection $\widehat{\nabla}$ satisfies the conditions in Theorem \ref{thm2}, then we only need to prove
the uniqueness. Let $\nabla^*$ be the other connection which satisfies (\ref{a9}) and (\ref{a10}). And let $\nabla^*_XY=\nabla^L_XY+B(X,Y),$
then\\
\begin{equation}\label{a11}
B(fX,Y)=fB(X,Y),~~B(X,fY)=(-1)^{|f||X|}B(X,Y).
\end{equation}
By $\nabla^L$ preserving the metric and (\ref{a10}), we get
\begin{align}\label{a11*}
&g(\nabla^*_XY,Z)+(-1)^{|X||Y|}g(Y,\nabla^*_XZ)\nonumber\\
&=g(\nabla^L_XY,Z)+g(B(X,Y),Z)+(-1)^{|X||Y|}g(Y,\nabla^L_XZ)+(-1)^{|X||Y|}g(Y,B(X,Z))\nonumber\\
&=X\left<Y,Z\right>_g+(-1)^{|Y||X|}g(Y,P)g(X,Z)+(-1)^{|X||Y|}(-1)^{|Z|(|X|+|Y|)}g(Z,P)g(Y,X).\nonumber\\
\end{align}
So
\begin{align}\label{a12}
&g(B(X,Y),Z)+(-1)^{|X||Y|}g(Y,B(X,Z))\nonumber\\
&=(-1)^{|Y||X|}g(Y,P)g(X,Z)+(-1)^{|X||Y|}(-1)^{|Z|(|X|+|Y|)}g(Z,P)g(Y,X).\nonumber\\
\end{align}
By $\nabla^L$ having no torsion, we have
\begin{align}\label{a13}
T_{\nabla^*}(X,Y)&=\nabla^*_XY-(-1)^{|X||Y|}\nabla^*_YX-[X,Y]\nonumber\\
&=\nabla^L_XY+B(X,Y)-(-1)^{|X||Y|}\nabla^L_YX-(-1)^{|X||Y|}B(Y,X)-[X,Y]\nonumber\\
&=B(X,Y)-(-1)^{|X||Y|}B(Y,X).
\end{align}
By (\ref{a12}) and (\ref{a13}) and $|B|=0$, we have
\begin{align}\label{a14}
&g(T_{\nabla^*}(X,Y),Z)+(-1)^{|Z|(|X|+|Y|)}g(T_{\nabla^*}(Z,X),Y)+(-1)^{|X||Y|}(-1)^{|Z|(|X|+|Y|)}g(T_{\nabla^*}(Z,Y),X)\nonumber\\
&=2g(B(X,Y),Z)-2(-1)^{|X||Y|}(-1)^{|Z|(|X|+|Y|)}g(Z,P)g(Y,X).\nonumber\\
\end{align}
By (\ref{a9}) and (\ref{a14}), we get
\begin{equation}\label{a15}
2g(B(X,Y),Z)=2g(X\cdot g(Y,P),Z),\nonumber\\
\end{equation}
then $B(X,Y)=X\cdot g(Y,P)$. So $\nabla^*=\widehat{\nabla}$, we get the proof of uniqueness.
\end{proof}
\begin{prop}\label{prop1}The following equality holds
\begin{align}\label{a16}
R_{\widehat{\nabla}}(X,Y)Z&=R^L(X,Y)Z+(-1)^{(|X|+|Y|)|Z|}[g(Z,\nabla^L_XP)Y
-(-1)^{|X||Y|}g(Z,\nabla^L_YP)X]\nonumber\\
&+(-1)^{(|X|+|Y|)|Z|}\pi(Z)[(-1)^{|X||Y|}\pi(Y)X
-\pi(X)Y],\nonumber\\
\end{align}
where $\pi$ be a one form defined by $\pi(Z):=g(Z,P)$ and $|\pi|=0$.\\
\end{prop}
\begin{proof}
By Definition \ref{def6} and Definition \ref{def11}, we have
\begin{align}\label{aaa}
R_{\widehat{\nabla}}(X,Y)Z&=\nabla^L_X\nabla^L_YZ+(-1)^{|Y||Z|}\nabla^L_X(\pi(Z)Y)-(-1)^{|X||Y|}[\nabla^L_Y\nabla^L_XZ+(-1)^{|X||Z|}\nabla^L_Y(\pi(Z)X)]\nonumber\\
&+(-1)^{|X|(|Y|+|Z|)}\pi(\nabla^L_YZ)X+(-1)^{|X|(|Y|+|Z|)}(-1)^{|Y||Z|}\pi(Z)\pi(Y)X\nonumber\\
&-(-1)^{|X||Y|}(-1)^{|Y|(|X|+|Z|)}[\pi(\nabla^L_XZ)Y+(-1)^{|X||Z|}\pi(Z)\pi(X)Y]-\nabla^L_{[X,Y]}Z\nonumber\\
&-(-1)^{(|X|+|Y|)|Z|}\pi(Z)[X,Y],\nonumber\\
\end{align}
by $\nabla^L$ preserving metric, we have
\begin{equation}\label{a17}
\nabla^L_X(\pi(Z)Y)=\pi(\nabla^L_XZ)Y+(-1)^{|X||Z|}g(Z,\nabla^L_XP)Y+(-1)^{|X||Z|}\pi(Z)\nabla^L_XY.
\end{equation}
Then, bring (\ref{a17}) into (\ref{aaa}), we can get Proposition \ref{prop1}.
\end{proof}

\section{Super warped products with a semi-symmetric non-metric connection}
\label{Section:3}
Let $(M=M_1\times_\mu M_2,g_\mu=\pi^*_1 g_1+\pi^*_1(\mu)\pi_2^*g_2)$ be the super warped product with $|g|=|g_1|=|g_2|$ and $|\mu|=0$. For simplicity, we assume that $\mu=h^2$ with $|h|=0$. Let $\nabla^{L,\mu}$ be the Levi-Civita connection on $(M,g_\mu)$ and $\nabla^{L,M_1}$ (resp. $\nabla^{L,M_2}$) be the Levi-Civita connection on $(M_1,g_1)$ (resp. $(M_2,g_2)$).
\begin{lem}\cite{WY}\label{lem1}
For $X,Y,Z\in{\rm Vect}(M_1)$ and $U,W,V\in {\rm Vect}(M_2)$, we have
\begin{align}\label{b1}
&(1)\nabla^{L,\mu}_XY=\nabla^{L,M_1}_XY,~~(2)\nabla^{L,\mu}_XU=\frac{X(h)}{h}U,\nonumber\\
&(3)\nabla^{L,\mu}_UX=(-1)^{|U||X|}\frac{X(h)}{h}U,~~(4)\nabla^{L,\mu}_UW=-hg_2(U,W){\rm grad}_{g_1}h+\nabla^{L,M_2}_UW.\nonumber\\
\end{align}
\end{lem}
Let $R^{L,\mu}$ denote the curvature tensor of the Levi-Civita connection on $(M,g_\mu)$ and let $R^{L,M_1}$ (resp. $R^{L,M_2}$) be the curvature tensor of the Levi-Civita connection on $(M_1,g_1)$ (resp. $(M_2,g_2)$). Let $H^h_{M_1}(X,Y):=XY(h)-\nabla^{L,{M_1}}_XY(h)$, then
$H^h_{M_1}(fX,Y)=fH^h_{M_1}(X,Y)$ and $H^h_{M_1}(X,fY)=(-1)^{|f||X|}fH^h_{M_1}(X,Y)$. $H^h_{M_1}$ is a $(0,2)$ tensor.
\begin{prop}\cite{WY}\label{prop2}
For $X,Y,Z\in{\rm Vect}(M_1)$ and $U,V,W\in {\rm Vect}(M_2)$, we have
\begin{align}\label{b3}
&(1)R^{L,\mu}(X,Y)Z=R^{L,M_1}(X,Y)Z,~~(2)R^{L,\mu}(V,X)Y=-(-1)^{|V|(|X|+|Y|)}\frac{H^h_{M_1}(X,Y)}{h}V,\nonumber\\
&(3)R^{L,\mu}(X,Y)V=0,~~(4)R^{L,\mu}(V,W)X=0,\nonumber\\
&(5)R^{L,\mu}(X,V)W=-(-1)^{|X|(|V|+|W|+|g|)}\frac{g_\mu(V,W)}{h}\nabla^{L,M_1}_X({\rm grad}_{g_1}h),\nonumber\\
&(6)R^{L,\mu}(V,W)U=R^{L,M_2}(V,W)U-(-1)^{|V|(|W|+|U|)}g_2(W,U)({\rm grad}_{g_1}h)(h)V\nonumber\\
&+(-1)^{|W||U|}g_2(V,U)({\rm grad}_{g_1}h)(h)W.\nonumber\\
\end{align}
\end{prop}
For $\overline{X},\overline{Y},{P}\in {\rm Vect}(M)$, we define
\begin{equation}\label{b12}
\widehat{\nabla}^\mu_{\overline{X}}\overline{Y}=\nabla^{L,\mu}_{\overline{X}}\overline{Y}+\overline{X}\cdot g_\mu(\overline{Y},{P}).
\end{equation}
For ${X},{Y},{P}\in {\rm Vect}(M_1)$, we define
\begin{equation}\label{b13}
\widehat{\nabla}^{M_1}_{X}{Y}=\nabla^{L,M_1}_{X}{Y}+{X}\cdot g_1({Y},{P}).
\end{equation}
By Lemma \ref{lem1}, (\ref{b12}) and (\ref{b13}), we have
\begin{lem}\label{lem2}
For $X,Y,P\in{\rm Vect}(M_1)$ and $U,W\in {\rm Vect}(M_2)$ and $\pi(X)=g_1(X,P)$, we have
\begin{align}\label{b14}
&(1)\widehat{\nabla}^{\mu}_XY=\widehat{\nabla}^{M_1}_XY,~~(2)\widehat{\nabla}^{\mu}_XU=\frac{X(h)}{h}U,\nonumber\\
&(3)\widehat{\nabla}^{\mu}_UX=(-1)^{|U||X|}[\frac{X(h)}{h}+\pi(X)]U,\nonumber\\
&(4)\widehat{\nabla}^{\mu}_UW=-hg_2(U,W){\rm grad}_{g_1}h+\nabla^{L,N}_UW.\nonumber\\
\end{align}
\end{lem}
\begin{lem}\label{lem3}
For $X,Y\in{\rm Vect}(M_1)$ and $U,W,P\in {\rm Vect}(M_2)$, we have
\begin{align}\label{b15}
&(1)\widehat{\nabla}^{\mu}_XY=\nabla^{L,M_1}_XY-g_1(X,Y)P,~~(2)\widehat{\nabla}^{\mu}_XU=\frac{X(h)}{h}U+X\cdot g_\mu(U,P),\nonumber\\
&(3)\widehat{\nabla}^{\mu}_UX=(-1)^{|U||X|}\frac{X(h)}{h}U,\nonumber\\
&(4)\widehat{\nabla}^{\mu}_UW=-hg_2(U,W){\rm grad}_{g_1}h+\nabla^{L,M_2}_UW+U\cdot g_\mu(W,P).\nonumber\\
\end{align}
\end{lem}
By Proposition \ref{prop1}, Lemma \ref{lem2} and Lemma \ref{lem3}, we get the following propositions,
\begin{prop}\label{P}
For $X,Y,Z,P\in{\rm Vect}(M_1)$ and $U,V,W\in {\rm Vect}(M_2)$, we have
\begin{align}\label{b16}
(1)R_{\widehat{\nabla}^\mu}(X,Y)Z&=R_{\widehat{\nabla}^{M_1}}(X,Y)Z,\nonumber\\
(2)R_{\widehat{\nabla}^\mu}(V,X)Y&=-(-1)^{|V|(|X|+|Y|)}\left[\frac{H^h_{M_1}(X,Y)}{h}+(-1)^{|X||Y|}g_1(Y,\nabla^{L,M_1}_XP)-\pi(X)\pi(Y)\right]V,\nonumber\\
(3)R_{\widehat{\nabla}^\mu}(X,Y)V&=0,~~(4)R_{\widehat{\nabla}^\mu}(V,W)X=0,\nonumber\\
(5)R_{\widehat{\nabla}^\mu}(X,V)W&=-(-1)^{|X|(|V|+|W|+|g|)}{g_\mu(V,W)}\left[\frac{\nabla^{L,M_1}_X({\rm grad}_{g_1}h)}{h}+(-1)^{(|X|+|P|)|g|}\frac{P(h)}{h}X\right],\nonumber\\
{\rm when}~~|g|=|P|=0&,~~{\rm then}\nonumber\\
R_{\widehat{\nabla}^\mu}(X,V)W&=-(-1)^{|X|(|V|+|W|)}{g_\mu(V,W)}\left[\frac{\nabla^{L,M_1}_X({\rm grad}_{g_1}h)}{h}+\frac{P(h)}{h}X\right],\nonumber\\
(6)R_{\widehat{\nabla}^\mu}(U,V)W&=R^{L,M_2}(U,V)W+\left[(-1)^{|g|(|W|+|g|)}\frac{({\rm grad}_{g_1}h)(h)}{h^2}+(-1)^{|P|(|W|+|g|)}\frac{P(h)}{h}\right]\nonumber\\
&\cdot\left[(-1)^{|V||W|}(-1)^{|P||U|}g_\mu(U,W)V-(-1)^{|U|(|V|+|W|)}(-1)^{|P||V|}g_\mu(V,W)U\right],\nonumber\\
{\rm when}~~|g|=|P|=0&,~~{\rm then}\nonumber\\
R_{\widehat{\nabla}^\mu}(U,V)W&=R^{L,M_2}(U,V)W+\left[\frac{({\rm grad}_{g_1}h)(h)}{h^2}+\frac{P(h)}{h}\right]\nonumber\\
&\cdot\left[(-1)^{|V||W|}g_\mu(U,W)V-(-1)^{|U|(|V|+|W|)}g_\mu(V,W)U\right].\nonumber\\
\end{align}
\end{prop}
\begin{proof}
(1)By Lemma \ref{lem1} and Definition \ref{def6}, we get
\begin{align}\label{f1}
R_{\widehat{\nabla}^\mu}(X, Y )Z &={\widehat{\nabla}^\mu}_X{\widehat{\nabla}^\mu}_Y-(-1)^{|X||Y|}{\widehat{\nabla}^\mu}_Y{\widehat{\nabla}^\mu}_X-{\widehat{\nabla}^\mu}_{[X,Y]}Z,\nonumber\\
&={\widehat{\nabla}^{M_1}}_X{\widehat{\nabla}^{M_1}}_Y-(-1)^{|X||Y|}{\widehat{\nabla}^{M_1}}_Y{\widehat{\nabla}^{M_1}}_X-{\widehat{\nabla}^{M_1}}_{[X,Y]}Z,\nonumber\\
&=R_{\widehat{\nabla}^{M_1}}(X, Y )Z,\nonumber\\
\end{align}
 so (1) holds.\\
\indent (2)By Lemma \ref{lem1} and Proposition \ref{prop1}, we have
\begin{align}\label{f2}
R_{\widehat{\nabla}^\mu}(V,X)Y&=R^{L,\mu}(V,X)Y+(-1)^{(|V|+|X|)|Y|}[g_\mu(Y,\nabla^{L,\mu}_VP)X
-(-1)^{|X||V|}g_\mu(Y,\nabla^{L,\mu}_XP)V]\nonumber\\
&+(-1)^{(|X|+|V|)|Y|}\pi(Y)[(-1)^{|X||V|}\pi(X)V
-\pi(V)X],\nonumber\\
&=R^{L,\mu}(V,X)Y
+(-1)^{(|V|+|X|)|Y|}(-1)^{|X||V|}[\pi(Y)\pi(X)V-g_1(Y,\nabla^{L,M_1}_XP)V],\nonumber\\
&=-(-1)^{|V|(|X|+|Y|)}\left[\frac{H^h_{M_1}(X,Y)}{h}+(-1)^{|X||Y|}g_1(Y,\nabla^{L,M_1}_XP)-\pi(X)\pi(Y)\right]V,\nonumber\\
\end{align}
so we get (2).\\
\indent (3)By Lemma \ref{lem1} and Proposition \ref{prop1}, we have
\begin{align}\label{f3}
R_{\widehat{\nabla}^\mu}(X,Y)V&=R^{L,\mu}(X,Y)V+(-1)^{(|X|+|Y|)|V|}[g_\mu(V,\nabla^{L,\mu}_XP)Y
-(-1)^{|X||Y|}g_\mu(V,\nabla^{L,\mu}_YP)X]\nonumber\\
&+(-1)^{(|X|+|Y|)|V|}\pi(V)[(-1)^{|X||Y|}\pi(Y)X
-\pi(X)Y]\nonumber\\
&=0,\nonumber\\
\end{align}
then we get (3).\\
\indent(4)Similar to (3), we get $R_{\widehat{\nabla}^\mu}(V,W)X=0$.\\
\indent (5)By Lemma \ref{lem1} and Proposition \ref{prop1}, we have
\begin{align}\label{f4}
R_{\widehat{\nabla}^\mu}(X,V)W&=R^{L,\mu}(X,V)W+(-1)^{(|V|+|X|)|W|}[g_\mu(W,\nabla^{L,\mu}_XP)V
-(-1)^{|X||V|}g_\mu(W,\nabla^{L,\mu}_VP)X]\nonumber\\
&+(-1)^{(|X|+|V|)|W|}\pi(W)[(-1)^{|X||V|}\pi(V)X
-\pi(X)V],\nonumber\\
&=-(-1)^{(|V|+|W|+|g|)|X|}\frac{g_\mu(V,W)}{h}\nabla^{L,M_1}_X(grad_{g_1}h)\nonumber\\
&-(-1)^{|W|(|X|+|V|)}(-1)^{|X||V|}g_\mu(W,(-1)^{|P||V|}\frac{P(h)}{h}V)X\nonumber\\
\end{align}
and by
\begin{align}\label{f5}
&(-1)^{|W|(|X|+|V|)}(-1)^{|X||V|}g_\mu(W,(-1)^{|P||V|}\frac{P(h)}{h}V)X\nonumber\\
&=(-1)^{|X|(|W|+|V|+|g|)}(-1)^{(|X|+|P|)|g|}g_\mu(V,W)\frac{P(h)}{h}X ,\nonumber\\
\end{align}
so we have
\begin{align}\label{f6}
R_{\widehat{\nabla}^\mu}(X,V)W&=-(-1)^{|X|(|V|+|W|+|g|)}{g_\mu(V,W)}\left[\frac{\nabla^{L,M_1}_X({\rm grad}_{g_1}h)}{h}+(-1)^{(|X|+|P|)|g|}\frac{P(h)}{h}X\right].\nonumber\\
\end{align}
Obviously, we can get
\begin{align}\label{f7}
{\rm when}~~|g|=|P|=0&,~~{\rm then}\nonumber\\
R_{\widehat{\nabla}^\mu}(X,V)W&=-(-1)^{|X|(|V|+|W|)}{g_\mu(V,W)}\left[\frac{\nabla^{L,M_1}_X({\rm grad}_{g_1}h)}{h}+\frac{P(h)}{h}X\right].\nonumber\\
\end{align}
\indent (6)By Lemma \ref{lem1} and Proposition \ref{prop1}, we have
\begin{align}\label{f8}
R_{\widehat{\nabla}^\mu}(U,V)W&=R^{L,\mu}(U,V)W+(-1)^{(|V|+|U|)|W|}[g_\mu(W,\nabla^{L,\mu}_UP)V
-(-1)^{|U||V|}g_\mu(W,\nabla^{L,\mu}_VP)U]\nonumber\\
&+(-1)^{(|V|+|U|)|W|}\pi(W)[(-1)^{|U||V|}\pi(V)U
-\pi(U)V]\nonumber\\
&=R^{L,M_2}(U,V)W-(-1)^{(|V|+|W|)|U|}g_2(V,W)(grad_{g_1}h)(h)U+(-1)^{|W||V|}g_2(U,W)\nonumber\\
&(grad_{g_1}h)(h)V+(-1)^{(|V|+|U|)|W|}(-1)^{|U||P|}(-1)^{|W||P|}\frac{P(h)}{h}g_\mu(W,U)V\nonumber\\
&-(-1)^{(|V|+|U|)|W|}(-1)^{|V||U|}(-1)^{|V||P|}(-1)^{|W||P|}\frac{P(h)}{h}g_\mu(W,V)U,\nonumber\\
\end{align}
by
\begin{align}\label{f9}
&-(-1)^{(|V|+|W|)|U|}g_2(V,W)(grad_{g_1}h)(h)U+(-1)^{|W||V|}g_2(U,W)(grad_{g_1}h)(h)V\nonumber\\
&=(-1)^{|g_1|(|W|+|g_2|)}(grad_{g_1}h)(h)[-(-1)^{(|V|+|W|)|U|}(-1)^{|V||g_1|}g_2(V,W)U+(-1)^{|W||U|}(-1)^{|U||g_1|}g_2(U,W)V] \nonumber\\
\end{align}
and
\begin{align}\label{f09}
&(-1)^{(|V|+|U|)|W|}(-1)^{|U||P|}(-1)^{|W||P|}\frac{P(h)}{h}g_\mu(W,U)V-(-1)^{(|V|+|U|)|W|}\nonumber\\
&(-1)^{|V||U|}(-1)^{|V||P|}(-1)^{|W||P|}\frac{P(h)}{h}g_\mu(W,V)U\nonumber\\
&=(-1)^{|P|(|W|+|g|)}\frac{P(h)}{h}[(-1)^{|V||W|}(-1)^{|P||U|}g_\mu(U,W)V-(-1)^{(|W|+|V|)|U|}(-1)^{|P||V|}g_\mu(V,W)U] ,\nonumber\\
\end{align}
so we have
\begin{align}\label{f10}
R_{\widehat{\nabla}^\mu}(U,V)W&=R^{L,M_2}(U,V)W+\left[(-1)^{|g|(|W|+|g|)}\frac{({\rm grad}_{g_1}h)(h)}{h^2}+(-1)^{|P|(|W|+|g|)}\frac{P(h)}{h}\right]\nonumber\\
&\cdot\left[(-1)^{|V||W|}(-1)^{|P||U|}g_\mu(U,W)V-(-1)^{|U|(|V|+|W|)}(-1)^{|P||V|}g_\mu(V,W)U\right].\nonumber\\
\end{align}
Obviously, we can get
\begin{align}\label{f11}
&{\rm when}~~|g|=|P|=0,~~{\rm then}\nonumber\\
&R_{\widehat{\nabla}^\mu}(U,V)W\nonumber\\
&=R^{L,M_2}(U,V)W+\left[\frac{({\rm grad}_{g_1}h)(h)}{h^2}+\frac{P(h)}{h}\right]\cdot\left[(-1)^{|V||W|}g_\mu(U,W)V-(-1)^{|U|(|V|+|W|)}g_\mu(V,W)U\right].\nonumber\\
\end{align}
\end{proof}
Similarly, we have
\begin{prop}\label{prop10}
For $X,Y,Z\in{\rm Vect}(M_1)$ and $U,V,W,P\in {\rm Vect}(M_2)$, we have
\begin{align}\label{b17}
(1)R_{\widehat{\nabla}^\mu}(X,Y)Z&=R^{L,M_1}(X,Y)Z,\nonumber\\
(2)R_{\widehat{\nabla}^\mu}(V,X)Y&=-(-1)^{|V|(|X|+|Y|)}\frac{H^h_{M_1}(X,Y)}{h}V-(-1)^{|X||Y|}hg_2(V,P)g_1(Y,{\rm grad}_{g_1}h)X\nonumber\\
&-(-1)^{|g(X,Y)||V|}g_1(X,Y)[\nabla^{L,M_2}_VP-hg_2(V,P){\rm grad}_{g_1}h],\nonumber\\
(3)R_{\widehat{\nabla}^\mu}(X,Y)V&=(-1)^{(|X|+|Y|)|V|}\pi(V)[\frac{X(h)}{h}Y-(-1)^{|X||Y|}\frac{Y(h)}{h}X],\nonumber\\
(4)R_{\widehat{\nabla}^\mu}(V,W)X&=-(-1)^{|X||W|}hg_2(V,P)g_1(X,{\rm grad}_{g_1}h)W+(-1)^{(|V|+|X|)|W|}hg_2(W,P)g_1(X,{\rm grad}_{g_1}h)V,\nonumber\\
(5)R_{\widehat{\nabla}^\mu}(X,V)W&=-(-1)^{|X|(|V|+|W|+|g|)}\frac{{g_\mu(V,W)}}{h}{\nabla^{L,M_1}_X({\rm grad}_{g_1}h)}+(-1)^{(|X|+|V|)|W|}[(-1)^{|X||W|}\frac{X(h)}{h}\nonumber\\
&g_\mu(W,P)V-(-1)^{|X||V|}g_\mu(W,\nabla^{L,M_2}_VP)X]+(-1)^{(|X|+|V|)|W|}(-1)^{|X||V|}\pi(W)\pi(V)X,\nonumber\\
(6)R_{\widehat{\nabla}^\mu}(U,V)W&=R^{L,M_2}(U,V)W-(-1)^{|U|(|V|+|W|)}g_2(V,W)({\rm grad}_{g_1}h)(h)U+(-1)^{|V||W|}g_2(U,W)\nonumber\\
&({\rm grad}_{g_1}h)(h)V+(-1)^{(|U|+|V|)|W|}[g_\mu(W,\nabla^{L,M_2}_UP)V-(-1)^{|U||V|}g_\mu(W,\nabla^{L,M_2}_VP)U]\nonumber\\
&+(-1)^{(|U|+|V|)|W|}\pi(W)[(-1)^{|U||V|}\pi(V)U-\pi(U)V].\nonumber\\
\end{align}
\end{prop}
In the following, we compute the Ricci tensor of $M$. Let $M_1$ (resp. $M_2$) have the $(p,m)$ (resp. $(q,n)$) dimension. Let $\partial_{x^I}=\{\partial_{x^a},\partial_{\xi^A}\}$ (resp.
$\partial_{y^J}=\{\partial_{y^b},\partial_{\eta^B}\}$) denote
the natural tangent frames on $M_1$ (resp. $M_2$). Let ${\rm Ric}^{L,\mu}$ (resp. ${\rm Ric}^{L,M_1}$, ${\rm Ric}^{L,M_2}$)  denote the Ricci tensor of $(M,g_\mu)$ (resp. $(M_1,g_1)$, $(M_2,g_2)$).  Then by (\ref{a4}), (\ref{a7}) and (\ref{b3}), we have

\begin{prop}\label{prop12}
The following equalities holds
\begin{align}\label{b18}
(1){\rm Ric}^{L,\mu}(\partial_{x^I},\partial_{x^K})&={\rm Ric}^{L,M_1}(\partial_{x^I},\partial_{x^K})-\frac{(q-n)}{h}H^h_{M_1}(\partial_{x^I},\partial_{x^K}),\nonumber\\
(2){\rm Ric}^{L,\mu}(\partial_{x^I},\partial_{y^J})&={\rm Ric}^{L,\mu}(\partial_{y^J},\partial_{x^I})=0,\nonumber\\
(3){\rm Ric}^{L,\mu}(\partial_{y^L},\partial_{y^J})&={\rm Ric}^{L,M_2}(\partial_{y^L},\partial_{y^J})-g_\mu(\partial_{y^L},\partial_{y^J})\cdot[\frac{\triangle^L_{g_1}(h)}{h}+(q-n-1)\frac{({\rm grad}_{g_1}h)(h)}{h^2}].\nonumber\\
\end{align}
\end{prop}
Let ${\rm Ric}^{\widehat{\nabla}^\mu}$ (resp. ${\rm Ric}^{\widehat{\nabla}^{M_1}}$) denote the Ricci tensor of $(M,\widehat{\nabla}^\mu,g_\mu)$ (resp. $(M_1,\widehat{\nabla}^{M_1},g_1$). Then by Proposition \ref{P} and (\ref{a4}), (\ref{a6}), we have
\begin{prop}\label{prop3}
The following equalities holds
\begin{align}\label{b19}
(1){\rm Ric}^{\widehat{\nabla}^\mu}(\partial_{x^I},\partial_{x^K})&={\rm Ric}^{\widehat{\nabla}^{M_1}}(\partial_{x^I},\partial_{x^K})-(q-n)\bigg[
\frac{H^h_{M_1}(\partial_{x^I},\partial_{x^K})}{h}-\pi(\partial_{x^I})\pi(\partial_{x^K})\\
&+\frac{(-1)^{|\partial_{x^I}||\partial_{x^K}|}g_1(\partial_{x^K},\nabla^{L,M}_{\partial_{x^I}}P)}{2}
+\frac{g_1(\partial_{x^I},\nabla^{L,M_1}_{\partial_{x^K}}P)}{2}\bigg],\nonumber\\
(2){\rm Ric}^{\widehat{\nabla}^\mu}(\partial_{x^I},\partial_{y^J})&={\rm Ric}^{\widehat{\nabla}^\mu}(\partial_{y^J},\partial_{x^I})=0,\nonumber\\
{\rm when}~~|g|=|P|=0,~~{\rm then}&\nonumber\\
(3){\rm Ric}^{\widehat{\nabla}^\mu}(\partial_{y^L},\partial_{y^J})&={\rm Ric}^{L,M_2}(\partial_{y^L},\partial_{y^J})-g_\mu(\partial_{y^L},\partial_{y^J})
[\frac{\triangle^L_{g_1}(h)}{h}+(q-n-1)\frac{({\rm grad}_{g_1}h)(h)}{h^2}\nonumber\\
&+(q-n-1+p-m)\frac{P(h)}{h}].\nonumber\\
\end{align}
\end{prop}
\begin{proof}
(1)By Definition \ref{def7}, we have
\begin{align}\label{f17}
{\rm Ric}^{\widehat{\nabla}^\mu}(\partial_{x^I},\partial_{x^K})&=\sum_L(-1)^{|\partial_{x^L}|(|\partial_{x^L}|+|\partial_{x^I}|+|\partial_{x^K}|)}\frac{1}{2}[R_{\widehat{\nabla}^\mu}(\partial_{x^L},\partial_{x^I})\partial_{x^K}+(-1)^{|\partial_{x^I}||\partial_{x^K}|}R_{\widehat{\nabla}^\mu}(\partial_{x^L},\partial_{x^K})\partial_{x^I}]^L\nonumber\\
&+\sum_J(-1)^{|\partial_{y^J}|(|\partial_{y^J}|+|\partial_{x^I}|+|\partial_{x^K}|)}\frac{1}{2}[R_{\widehat{\nabla}^\mu}(\partial_{y^J},\partial_{x^I})\partial_{x^K}+(-1)^{|\partial_{x^I}||\partial_{x^K}|}R_{\widehat{\nabla}^\mu}(\partial_{y^J},\partial_{x^K})\partial_{x^I}]^J\nonumber\\
&=\sum_L(-1)^{|\partial_{x^L}|(|\partial_{x^L}|+|\partial_{x^I}|+|\partial_{x^K}|)}\frac{1}{2}[R_{\widehat{\nabla}^{M_1}}(\partial_{x^L},\partial_{x^I})\partial_{x^K}+(-1)^{|\partial_{x^I}||\partial_{x^K}|}R_{\widehat{\nabla}^{M_1}}(\partial_{x^L},\partial_{x^K})\partial_{x^I}]^L\nonumber\\
&+\sum_J(-1)^{|\partial_{y^J}|(|\partial_{y^J}|+|\partial_{x^I}|+|\partial_{x^K}|)}\frac{1}{2}[R_{\widehat{\nabla}^\mu}(\partial_{y^J},\partial_{x^I})\partial_{x^K}+(-1)^{|\partial_{x^I}||\partial_{x^K}|}R_{\widehat{\nabla}^\mu}(\partial_{y^J},\partial_{x^K})\partial_{x^I}]^J\nonumber\\
&={\rm Ric}^{\widehat{\nabla}^{M_1}}(\partial_{x^I},\partial_{x^K})-\sum_J(-1)^{|\partial_{y^J}||\partial_{y^J}|}\frac{1}{2}\bigg[\frac{H_{M_1}^h(\partial_{x^I},\partial_{x^K})}{h}+(-1)^{|\partial_{x^I}||\partial_{x^K}|}\frac{H_{M_1}^h(\partial_{x^K},\partial_{x^I})}{h}\nonumber\\
&+(-1)^{|\partial_{x^I}||\partial_{x^K}|}g_1(\partial_{x^K},\widehat{\nabla}^{L,M_1}_{\partial_{x^L}}P)+g_1(\partial_{x^L},\widehat{\nabla}^{L,M_1}_{\partial_{x^K}}P)-2\pi(\partial_{x^L})\pi(\partial_{x^K})\bigg]\nonumber\\
&={\rm Ric}^{\widehat{\nabla}^{M_1}}(\partial_{x^I},\partial_{x^K})-(q-n)\bigg[
\frac{H^h_{M_1}(\partial_{x^I},\partial_{x^K})}{h}+\frac{(-1)^{|\partial_{x^I}||\partial_{x^K}|}g_1(\partial_{x^K},\nabla^{L,M_1}_{\partial_{x^I}}P)}{2}\nonumber\\
&+\frac{g_1(\partial_{x^I},\nabla^{L,M_1}_{\partial_{x^K}}P)}{2}-\pi(\partial_{x^I})\pi(\partial_{x^K})\bigg],\nonumber\\
\end{align}
so (1) holds.\\
(2)Similar to (2) in Propsition \ref{prop12}, we get
\begin{align}\label{f18}
{\rm Ric}^{\widehat{\nabla}^\mu}(\partial_{x^I},\partial_{y^J})&={\rm Ric}^{\widehat{\nabla}^\mu}(\partial_{y^J},\partial_{x^I})=0.\nonumber\\
\end{align}
(3)By Definition \ref{def7}, we have
\begin{align}\label{f19}
{\rm Ric}^{\widehat{\nabla}^\mu}(\partial_{y^L},\partial_{y^J})&=\sum_I(-1)^{|\partial_{x^I}|(|\partial_{x^I}|+|\partial_{y^L}|+|\partial_{y^J}|)}\frac{1}{2}[R_{\widehat{\nabla}^\mu}(\partial_{x^I},\partial_{y^L})\partial_{y^J}+(-1)^{|\partial_{y^L}||\partial_{y^J}|}R_{\widehat{\nabla}^\mu}(\partial_{x^I},\partial_{y^J})\partial_{y^L}]^I\nonumber\\
&+\sum_K(-1)^{|\partial_{y^K}|(|\partial_{y^K}|+|\partial_{y^L}|+|\partial_{y^J}|)}\frac{1}{2}[R_{\widehat{\nabla}^\mu}(\partial_{y^K},\partial_{y^L})\partial_{y^J}+(-1)^{|\partial_{y^L}||\partial_{y^J}|}R_{\widehat{\nabla}^\mu}(\partial_{y^K},\partial_{y^J})\partial_{y^L}]^K\nonumber\\
&=\Delta_1+\Delta_2,\nonumber\\
\end{align}
where
\begin{align}\label{1}
&\Delta_1:=\sum_I(-1)^{|\partial_{x^I}|(|\partial_{x^I}|+|\partial_{y^L}|+|\partial_{y^J}|)}\frac{1}{2}[R_{\widehat{\nabla}^\mu}(\partial_{x^I},\partial_{y^L})\partial_{y^J}+(-1)^{|\partial_{y^L}||\partial_{y^J}|}R_{\widehat{\nabla}^\mu}(\partial_{x^I},\partial_{y^J})\partial_{y^L}]^I,\nonumber\\
&\Delta_2:=\sum_K(-1)^{|\partial_{y^K}|(|\partial_{y^K}|+|\partial_{y^L}|+|\partial_{y^J}|)}\frac{1}{2}[R_{\widehat{\nabla}^\mu}(\partial_{y^K},\partial_{y^L})\partial_{y^J}+(-1)^{|\partial_{y^L}||\partial_{y^J}|}R_{\widehat{\nabla}^\mu}(\partial_{y^K},\partial_{y^J})\partial_{y^L}]^K,\nonumber\\
\end{align}
by Propsition \ref{b3}, we have
\begin{align}\label{f20}
R_{\widehat{\nabla}^\mu}(\partial_{x^I},\partial_{y^L})\partial_{y^J}=-(-1)^{|\partial_{x^I}|(|\partial_{y^L}|+|g|+|\partial_{y^J}|)}g_\mu(\partial_{y^L},\partial_{y^J})\bigg[\frac{\nabla^{L,M_1}_{\partial_{x^I}}(grad_{g_1}h)}{h}+(-1)^{(|\partial_{x^I}|+|P|)|g|}\frac{P(h)}{h}\partial_{x^I}\bigg],\nonumber\\
\end{align}
then, we get
\begin{align}\label{f21}
\Delta_1&=-g_\mu(\partial_{y^L},\partial_{y^J})[\frac{\triangle^L_{g_1}(h)}{h}+\sum_I(-1)^{|\partial_{x^I}||\partial_{x^I}|}(-1)^{|P||g|}\frac{P(h)}{h}],\nonumber\\
\end{align}
\begin{align}\label{f22}
\Delta_2&=\sum_K(-1)^{|\partial_{y^K}|(|\partial_{y^K}|+|\partial_{y^L}|+|\partial_{y^J}|)}\frac{1}{2}\bigg\{R^{L,M_2}(\partial_{y^K},\partial_{y^L})\partial_{y^J}+\frac{(grad_{g_1}h)(h)}{h^2}\bigg[(-1)^{|\partial_{y^L}||\partial_{y^J}|}g_\mu(\partial_{y^K},\partial_{y^J})\delta^K_L\nonumber\\
&-(-1)^{|\partial_{y^K}|(|\partial_{y^L}|+|\partial_{y^J}|)}g_\mu(\partial_{y^L},\partial_{y^J})\bigg]\bigg\},\nonumber\\
&=Ric^{L,M_2}(\partial_{y^L},\partial_{y^J})-g_\mu(\partial_{y^L},\partial_{y^J})\bigg[\frac{\triangle_g^L(h)}{h}+(p-m)P(h)+(q-n-1)\frac{grad_{g_1}h}{h}\bigg],\nonumber\\
\end{align}
\begin{align}\label{f23}
&{\rm when}~~|g|=|P|=0,~~{\rm then}\nonumber\\
&\Delta_1+\Delta_2={\rm Ric}^{L,M_2}(\partial_{y^L},\partial_{y^J})-g_\mu(\partial_{y^L},\partial_{y^J})
[\frac{\triangle^L_{g_1}(h)}{h}+(q-n-1)\frac{({\rm grad}_{g_1}h)(h)}{h^2}+(q-n-1+p-m)\frac{P(h)}{h}],\nonumber\\
\end{align}
so (3) holds.\\
\end{proof}
\section{Special super warped products with a semi-symmetric non-metric connection}
\label{Section:4}
\indent In this section, we construct an Einstein super warped product with a semi-symmetric non-metric connection. Let $(M_2^{(q,n)},g_2)$ be a super Riemannian manifold and $\mathbb{R}^{(1,0)}$ be the real line. We consider the super Riemannian manifold $M=\mathbb{R}^{(1,0)}\times _\mu M_2^{(q,n)}$ and $g_\mu=-dt\otimes dt+h^2g_2$, where
$h(t)$ and $\mu(t)=h(t)^2$ be non-zero functions for $t\in \mathbb{R}$ and $|g_2|=0$.\\
\indent Let $P=\partial_t$, then by Definition \ref{def6} and Definition \ref{def7}, we get $R_{\widehat{\nabla}^{\mathbb{R}}}(\partial_t,\partial_t)\partial_t=0$  and ${\rm Ric}^{\widehat{\nabla}^{\mathbb{R}}}(\partial_{t},\partial_{t})=0$. By computations, we have $H^h_{M_1}(\partial_{t},\partial_{t})=h''$, ${\rm grad}_{g_1}(h)=-h'\partial_t$ and $\triangle_{g_1}^L(h)=-h''$. By Propsition \ref{prop3}, we have
\begin{prop}\label{cprop1}
The following equalities holds
\begin{align}\label{c1}
(1){\rm Ric}^{\widehat{\nabla}^\mu}(\partial_{t},\partial_{t})&=-(q-n)(\frac{h''}{h}-1),\nonumber\\
(2){\rm Ric}^{\widehat{\nabla}^\mu}(\partial_{t},\partial_{y^J})&={\rm Ric}^{\widehat{\nabla}^\mu}(\partial_{y^J},\partial_{t})=0,\nonumber\\
(3){\rm Ric}^{\widehat{\nabla}^\mu}(\partial_{y^L},\partial_{y^J})&={\rm Ric}^{L,M_2}(\partial_{y^L},\partial_{y^J})-g_\mu(\partial_{y^L},\partial_{y^J})\cdot[-\frac{h''}{h}-(q-n-1)\frac{(h')^2}{h^2}+(q-n)\frac{h'}{h}].\nonumber\\
\end{align}
\end{prop}
\begin{proof}
(1)By (1) in Propsition \ref{prop3}, we have
\begin{align}\label{e1}
{\rm Ric}^{\widehat{\nabla}^\mu}(\partial_t,\partial_t)&={\rm Ric}^{\widehat{\nabla}^R}(\partial_t,\partial_t)-(q-n)\bigg[
\frac{H^h_{M_1}(\partial_t,\partial_t)}{h}-\pi(\partial_t)\pi(\partial_t)\nonumber\\
&+\frac{(-1)^{|\partial_t||\partial_t|}g_1(\partial_t,\nabla^{L,R}_{\partial_t}\partial_t)}{2}
+\frac{g_1(\partial_t,\nabla^{L,R}_{\partial_t}\partial_t)}{2}\bigg]\nonumber\\
&=-(q-n)(\frac{h''}{h}-1).\nonumber\\
\end{align}
(2)By (2) in Propsition \ref{prop3}, we have
\begin{align}\label{e2}
{\rm Ric}^{\widehat{\nabla}^\mu}(\partial_{t},\partial_{y^J})&={\rm Ric}^{\widehat{\nabla}^\mu}(\partial_{y^J},\partial_{t})=0.\nonumber\\
\end{align}
(3)By (3) in Propsition \ref{prop3}, $({\rm grad}_{g_1}h)(h)=-h'\partial_t$ and $\triangle^L_{g_1}(h)=-h''$, we have
\begin{align}\label{e3}
{\rm Ric}^{\widehat{\nabla}^\mu}(\partial_{y^L},\partial_{y^J})&={\rm Ric}^{L,M_2}(\partial_{y^L},\partial_{y^J})-g_\mu(\partial_{y^L},\partial_{y^J})
[\frac{\triangle^L_{g_1}(h)}{h}+(q-n-1)\frac{({\rm grad}_{g_1}h)(h)}{h^2}\nonumber\\
&+(q-n-1+p-m)\frac{P(h)}{h}]\nonumber\\
&={\rm Ric}^{L,M_2}(\partial_{y^L},\partial_{y^J})-g_\mu(\partial_{y^L},\partial_{y^J})
[\frac{-h''}{h}-(q-n-1)\frac{h'^2}{h^2}+(q-n)\frac{h'}{h}].\nonumber\\
\end{align}
\end{proof}
\begin{defn}\label{ee1}
We call that $(M,g_\mu,\widehat{\nabla}^\mu)$ is Einstein if
${\rm Ric}^{\widehat{\nabla}^\mu}(\overline{X},\overline{Y})=\lambda g_\mu(\overline{X},\overline{Y}),$ for $\overline{X},\overline{Y}\in {\rm Vect}(M)$ and a constant $\lambda$.
 \end{defn}
As in the ordinary warped product case (see Theorem 15 in \cite{W1}), by (\ref{c1}) and Definition \ref{ee1}, we have
the following theorems
\begin{thm}\label{cthm1} Let $M=\mathbb{R}^{(1,0)}\times _\mu M_2^{(q,n)}$ and $g_\mu=-dt\otimes dt+h^2g_2$ and $P=\partial_t$. Then
    $(M,g_\mu,\widehat{\nabla}^\mu)$ is Einstein with the Einstein constant
    $\lambda$ if and only if the following conditions are satisfied\\
    \noindent (1) $( M_2^{(q,n)},\nabla^{L,M_2})$ is Einstein with the Einstein constant $c_0$.\\
\noindent (2)
\begin{equation}\label{c3}
(q-n)(\frac{h''}{h}-1)=\lambda.
\end{equation}
(3)
\begin{equation}\label{c4}
\lambda h^2-h''h-(q-n-1)(h')^2+(q-n)hh'=c_0.
\end{equation}
\end{thm}
\begin{proof}
(1)By (3) in Propsition \ref{cprop1}, we have
\begin{align}\label{e4}
{\rm Ric}^{\widehat{\nabla}^\mu}(\partial_{y^L},\partial_{y^J})&={\rm Ric}^{L,M_2}(\partial_{y^L},\partial_{y^J})-g_\mu(\partial_{y^L},\partial_{y^J})\cdot[-\frac{h''}{h}-(q-n-1)\frac{(h')^2}{h^2}+(q-n)\frac{h'}{h}],\nonumber\\
\end{align}
then
\begin{align}\label{e5}
{\rm Ric}^{L,M_2}(\partial_{y^L},\partial_{y^J})&={\rm Ric}^{\widehat{\nabla}^\mu}(\partial_{y^L},\partial_{y^J})+g_\mu(\partial_{y^L},\partial_{y^J})\cdot[-\frac{h''}{h}-(q-n-1)\frac{(h')^2}{h^2}+(q-n)\frac{h'}{h}]\nonumber\\
&=\lambda g_\mu(\partial_{y^L},\partial_{y^J})+h^2g_2(\partial_{y^L},\partial_{y^J})\cdot[-\frac{h''}{h}-(q-n-1)\frac{(h')^2}{h^2}+(q-n)\frac{h'}{h}]\nonumber\\
&=\lambda h^2 g_2(\partial_{y^L},\partial_{y^J})+h^2g_2(\partial_{y^L},\partial_{y^J})\cdot[-\frac{h''}{h}-(q-n-1)\frac{(h')^2}{h^2}+(q-n)\frac{h'}{h}]\nonumber\\
&=h^2\left(\lambda-\frac{h''}{h}-(q-n-1)\frac{(h')^2}{h^2}+(q-n)\frac{h'}{h}\right)g_2(\partial_{y^L},\partial_{y^J})\nonumber\\
&=L(t)g_2(\partial_{y^L},\partial_{y^J}),\nonumber\\
\end{align}
by two sides of the equation (\ref{e5}) act simultaneously on $\partial_t$ and $g_2(\partial_{y^L},\partial_{y^J})\neq 0$, we have $L(t)=c_0$, so ${\rm Ric}^{L,M_2}(\partial_{y^L},\partial_{y^J})=c_0g_2(\partial_{y^L},\partial_{y^J})$, therefore (1) holds.\\
(2)By (1) in Propsition \ref{cprop1} and Definition \ref{ee1}, we have
\begin{align}\label{e6}
{\rm Ric}^{\widehat{\nabla}^\mu}(\partial_{t},\partial_t)&=\lambda g_\mu(\partial_t,\partial_t)=-\lambda=-(q-n)(\frac{h''}{h}-1),\nonumber\\
\end{align}
then we get $\lambda=(q-n)(\frac{h''}{h}-1).$\\
(3)By (1), we get $\lambda h^2-h''h-(q-n-1)(h')^2+(q-n)hh'=c_0$.\\
\end{proof}
By Theorem \ref{cthm1}, similar to the ordinary warped product case (see Theorem 3.1 in \cite{W01}), we have
\begin{thm}\label{cthm2}
Let $M=\mathbb{R}^{(1,0)}\times _\mu M_2^{(q,n)}$ and $g_\mu=-dt\otimes dt+h^2g_2$ and $P=\partial_t,$ when $q-n=1$, then
    $(M,g_\mu,\nabla^\mu)$ is Einstein with the Einstein constant
    $-\lambda_0$ if and only if the following conditions are satisfied\\
    \noindent (1) $( M_2^{(q,n)},\nabla^{L,M_2})$ is Einstein with the Einstein constant $c_0=hh'-h^2$.\\
\noindent (2-1) $\lambda_0<1,~~f(t)=c_1e^{\sqrt{1-\lambda_0}t}+c_2e^{-\sqrt{1-\lambda_0}t},$\\
 \noindent (2-2)
$\lambda_0=1,~~f(t)=c_1+c_2t,$\\
\noindent (2-3) $\lambda_0>1,~~f(t)=c_1{\rm
cos}\left(\sqrt{\lambda_0-1}t\right)+ c_2{\rm sin}\left(\sqrt{\lambda_0-1}t\right),$\\
\end{thm}
\begin{proof}
(1)Let $\lambda=-\lambda_0$ and $c_0=-\lambda_N$, then
\begin{align}\label{e7}
\lambda_N-hh''-(q-n-1)h'^2-\lambda_0h^2+(q-n)hh'=0,\nonumber\\
\end{align}
when $q-n=1$, then\nonumber\\
\begin{align}\label{e8}
\lambda_N-hh''-\lambda_0h^2+hh'=0.\nonumber\\
\end{align}
By $(q-n)(\frac{h''}{h}-1)=-\lambda_0,$ we have $\lambda_N=h^2-hh'$ and $h''=(1-\lambda_0)h$,
so
$( M_2^{(q,n)},\nabla^{L,M_2})$ is Einstein with the Einstein constant $c_0=hh'-h^2$.\\
(2)By $h''=(1-\lambda_0)h$, we have characteristic equation $\mu^2-(1-\lambda_0)=0,$ then $\Delta=4(1-\lambda_0)$, so (2) holds.
\end{proof}
\begin{prop}\label{cprop2}
Let $M=\mathbb{R}^{(1,0)}\times _\mu M_2^{(q,n)}$ and $g_\mu=-dt\otimes dt+h^2g_2$ and $P=\partial_t,$ when $q-n=0$, then
    $(M,g_\mu,\widehat{\nabla}^\mu)$ is Einstein with the Einstein constant
    $-\lambda_0$ if and only if the following conditions are satisfied\\
    \noindent (1) $(M_2^{(q,n)},\nabla^{L,M_2})$ is Einstein with the Einstein constant $c_0$.\\
\noindent (2)$\lambda_0=0,$\\
\noindent (3) $c_0+hh''-h'^2=0$\\
\end{prop}
\begin{proof}
When $q-n=0,$ we have $\lambda_0=0$ and $\lambda_N-hh''+h'^2=0$, then we get Propsition \ref{cprop2}.\\
\end{proof}
\begin{thm}\label{cthm3}
Let $M=\mathbb{R}^{(1,0)}\times _\mu M_2^{(q,n)}$ and $g_\mu=-dt\otimes dt+h^2g_2$ and $P=\partial_t$, when $q-n\neq 0,1$, then
    $(M,g_\mu,\widehat{\nabla}^\mu)$ is Einstein with the Einstein constant
    $-\lambda_0$ if and only if $(M_2^{(q,n)},\nabla^{L,M_2})$ is Einstein with the Einstein constant $c_0$ and one of the following conditions
holds\\
(1)$\lambda_0=\lambda_N=0, h=c_1e^t;$\\
(2)$\lambda_0=q-n, h=c_1=\sqrt{\frac{\lambda_N}{q-n}}.$
\end{thm}
\begin{proof}
Let $\lambda=-\lambda_0$ and $c_0=-\lambda_N$, then by (\ref{c3}), we have $h''=(1-\frac{\lambda_0}{q-n})h$. By (\ref{c4}), we get
\begin{align}\label{www}
\lambda_N-(q-n-1)h'^2+(q-n)hh'-(1+\lambda_0-\frac{\lambda_0}{q-n})h^2=0,\nonumber\\
\end{align}
when $q-n\neq0,1,$ we get
\begin{align}\label{ww}
\frac{\lambda_N}{1-q+n}+h'^2+\frac{(q-n)}{1-q+n}hh'+(\frac{\lambda_0}{q-n}-\frac{1}{1-q+n})h^2=0.\nonumber\\
\end{align}
Let $q-n=l,$ $\frac{\lambda_0}{q-n}=\frac{\lambda_0}{l}=d_0,$ $\frac{\lambda_N}{1-q+n}=\frac{\lambda_N}{1-L}=\overline{d_0},$\\
 {\bf case(a)}when $d_0<1,$ let $a_0=\sqrt{1-d_0},$ $b_0=-\sqrt{1-d_0},$ then $a_0+b_0=0,$ $a_0b_0=d_0-1$ and $h=c_1e^{a_0t}+c_2e^{b_0t},$ by (\ref{ww}), we have
\begin{align}\label{ss}
&\overline{d_0}+c_1^2(a_0^2+a_0b_0+1-\frac{1}{1-l}+\frac{l}{1-l}a_0)e^{2a_0t}+c_2^2(b_0^2+a_0b_0+1-\frac{1}{1-l}+\frac{l}{1-l}b_0)e^{2b_0t}\nonumber\\
&+c_1c_2(4a_0b_0-\frac{l}{1-l})=0,\nonumber\\
\end{align}
then
\begin{align}\label{ddd}
&\overline{d_0}+c_1c_2(4a_0b_0-\frac{l}{1-l})=0,\nonumber\\
&c_1^2(a_0^2+a_0b_0+1-\frac{1}{1-l}+\frac{l}{1-l}a_0)=0,\nonumber\\
&c_2^2(b_0^2+a_0b_0+1-\frac{1}{1-l}+\frac{l}{1-l}b_0)=0.\nonumber\\
\end{align}
 {\bf case(a-1)}When $c_1=0,~~c_2\neq 0,$ we get $a_0=-1,~~b_0=1,$ then this is a contradiction.\\
 {\bf case(a-2)}When $c_1\neq0,~~c_2=0,$ we get $a_0=1,~~b_0=-1,~~\lambda_N=0,~~\lambda_0=0,~~h=c_1e^t.$\\
 {\bf case(a-3)}When $c_1\neq0,~~c_2\neq0,$ then there is no solution.\\

\noindent{\bf case(b)}When $d_0=1,$ then $h=c_1+c_2t,$ by (\ref{ww}), we have
\begin{align}\label{s1s}
&\overline{d_0}+c_1^2(d_0-\frac{1}{1-l})+c_2^2[1+(d_0-\frac{1}{1-l})t^2+\frac{l}{1-l}t]+c_1c_2[2(d_0-\frac{1}{1-l})t+\frac{l}{1-l}]=0.\nonumber\\
\end{align}
 {\bf case(b-1)}When $c_1=0,~~c_2\neq 0,$ then $\overline{d_0}+c_2^2[1+(d_0-\frac{1}{1-l})t^2+\frac{l}{1-l}t]=0,$ we get $c_2=0,$  this is a contradiction.\\
 {\bf case(b-2)}When $c_1\neq0,~~c_2=0,$ then $\overline{d_0}+c_1^2(d_0-\frac{1}{1-l})=0,$ we get $h=c_1=\frac{\lambda_N}{l}.$\\
 {\bf case(b-3)}When $c_1\neq0,~~c_2\neq0,$ then $c_2^2(d_0-\frac{1}{1-l})=0,~~c_2^2\frac{l}{1-l}+2c_1c_2(d_0-\frac{1}{1-l})=0,$ we get $c_2=0,$ so this is a contradiction.\\

 \noindent{\bf case(c)}When $d_0>1,$ let $h_0=\sqrt{d_0-1},$ then $h=c_1cosh_0t+c_2sinh_0t,$ by (\ref{ww}), we have
\begin{align}\label{s2s}
&\overline{d_0}+(sinh_0t)^2[c_1^2h_0^2+c_2(d_0-\frac{1}{1-l})-c_1c_2\frac{l}{1-l}h_0]+(cosh_0t)^2[c_2^2h_0^2+c_1(d_0-\frac{1}{1-l})+c_1c_2\frac{l}{1-l}h_0]\nonumber\\
&+cosh_0tsinh_0t[-2c_1c_2h_0^2+2c_1c_2(d_0-\frac{l}{1-l})-c_1^2h_0\frac{1}{1-l}+c_2^2h_0\frac{1}{1-l}]=0,\nonumber\\
\end{align}
then\\
\begin{align}\label{s3s}
&\overline{d_0}+c_1^2h_0^2+c_2(d_0-\frac{1}{1-l})-c_1c_2\frac{l}{1-l}h_0=0,\nonumber\\
&\overline{d_0}+c_2^2h_0^2+c_1(d_0-\frac{1}{1-l})+c_1c_2\frac{l}{1-l}h_0=0,\nonumber\\
&-2c_1c_2h_0^2+2c_1c_2(d_0-\frac{l}{1-l})-c_1^2h_0\frac{1}{1-l}+c_2^2h_0\frac{1}{1-l}=0.\nonumber\\
\end{align}
By (\ref{s3s}), we can get $c_1=c_2=0,$ so this is a contradiction.\\
\end{proof}
 \indent Nextly, we give another example. Let $M_1=\mathbb{R}^{(1,2)}$ with coordinates $(t,\xi,\eta)$ and $|t|=0,~~|\xi|=|\eta|=1$.
 We give a metric $g_1=-dt\otimes dt+d\xi\otimes d\eta-d\eta\otimes d\xi$ on $M_1$ i.e.
\begin{equation}\label{c5}
g_1(\partial_t,\partial_t)=-1,~~g_1(\partial_\xi,\partial_\eta)=-1,~~g_1(\partial_\eta,\partial_\xi)=1,~~g_1(\partial_{x^I},\partial_{x^K})=0,
\end{equation}
for the other pair $(\partial_{x^I},\partial_{x^K})$. Let $\widetilde{M}=\mathbb{R}^{(1,2)}\times_\mu M_2^{(q,n)}$ and $g_\mu=g_1+h(t)^2g_2$ and $P=\partial_t$. By Proposition 7 in \cite{BG}, we have the Christoffel symbols $\Gamma^L_{JI}=0$, then
\begin{equation}\label{c6}
\nabla^{L,g_1}_{\partial_{x^J}}\partial_{x^K}=0,~~ R^{L,g_1}(X,Y)Z=0,~~{\rm Ric}^{L,g_1}(X,Y)=0.
\end{equation}
We have
\begin{equation}\label{c7}
H^h_{M_1}(\partial_{t},\partial_{t})=h'',~~H^h_M(\partial_{x^J},\partial_{x^K})=0,~~{\rm for~~ the~~ other~~ pair~~} (\partial_{x^I},\partial_{x^K}).
\end{equation}
\begin{equation}\label{c8}
 {\rm grad}_{g_1}(h)=-h'\partial_t,~~ \triangle_{g_1}^L(h)=-h''.
\end{equation}
By Proposition \ref{prop12} and the Einstein condition, we have
\begin{thm}\label{cthm5}
Let $\widetilde{M}=\mathbb{R}^{(1,2)}\times _\mu M_2^{(q,n)}$ and $g_\mu=g_1+h^2g_2$ and $P=\partial_t$. Then
    $(\widetilde{M},g_\mu,\nabla^{L,\mu})$ is Einstein with the Einstein constant
    $\lambda$ if and only if one of the following conditions is satisfied\\
    \noindent (1) $\lambda=0$, $q=n$, $(M_2^{(q,n)},\nabla^{L,M_2})$ is Einstein with the Einstein constant $-c_0$ and $hh''-h'^2=c_0$.\\
\noindent (2) $\lambda=0$, $q-n-1=0$, $(M_2^{(q,n)},\nabla^{L,M_2})$ is Einstein with the Einstein constant $0$ and $h=c_1t+c_2$ where $c_1,c_2$ are constant.\\
\noindent (3) $\lambda=0$, $q-n-1\neq 0,-1$, $(M_2^{(q,n)},\nabla^{L,M_2})$ is Einstein with the Einstein constant $-c_0$ and
$h=\pm \sqrt{\frac{c_0}{q-n-1}}t+c_2$, $\frac{c_0}{q-n-1}\geq 0$.\\
\end{thm}
\begin{proof}
By (1) in Propsition \ref{prop12}, we have
\begin{align}\label{k1}
{\rm Ric}^{L,\mu}(\partial_{x^I},\partial_{x^K})&={\rm Ric}^{L,M_1}(\partial_{x^I},\partial_{x^K})-\frac{(q-n)}{h}H^h_{M_1}(\partial_{x^I},\partial_{x^K}),\nonumber\\
\end{align}
then
\begin{align}\label{k2}
\lambda g_\mu(\partial_{x^I},\partial_{x^K})&=-\frac{(q-n)}{h}H^h_{M_1}(\partial_{x^I},\partial_{x^K}),\nonumber\\
\end{align}
so we get $\lambda=0$ and $q=n$ or $h''=0.$\\
By $\lambda=0$, then we have\\
\begin{align}\label{jjj}
{\rm Ric}^{L,\mu}(\partial_{x^I},\partial_{y^J})&={\rm Ric}^{L,\mu}(\partial_{y^J},\partial_{x^I})=0.\nonumber\\
\end{align}
By (3) in Propsition \ref{prop12} and (\ref{jjj}), we have
\begin{align}\label{k3}
{\rm Ric}^{L,\mu}(\partial_{y^I},\partial_{y^J})&=g_2(\partial_{y^I},\partial_{y^J})[-hh''-(q-n-1)h'^2].\nonumber\\
\end{align}
Then we get\\
{\bf(case-a)}when $\lambda=0, q=n,$ by $hh''+(q-n-1)h'^2=c_0,$ we get $(M_2^{(q,n)},\nabla^{L,M_2})$ is Einstein with the Einstein constant $-c_0$ and $hh''-h'^2=c_0$,\\
{\bf(case-b)}when $\lambda=0, h''=0,$ by $hh''+(q-n-1)h'^2=c_0,$ we have $(q-n-1)h'^2=c_0,$\\
{\bf(case-b-1)}when $q-n-1=0,$ $(M_2^{(q,n)},\nabla^{L,M_2})$ is Einstein with the Einstein constant $0$ and $h=c_1t+c_2$ where $c_1,c_2$ are constant,\\
{\bf(case-b-2)}when $q\neq n,$ $q-n-1\neq0,$ $(M_2^{(q,n)},\nabla^{L,M_2})$ is Einstein with the Einstein constant $-c_0$ and
$h=\pm \sqrt{\frac{c_0}{q-n-1}}t+c_2$, $\frac{c_0}{q-n-1}\geq 0$.\\
\end{proof}
\indent By (\ref{a16}) and (\ref{c6}), we can get
\begin{align}\label{c9}
&{R}^{\widehat{\nabla}^{\mathbb{R}^{(1,2)}}}(\partial_t,\partial_\xi)\partial_t=-\partial_\xi,~~~
{R}^{\widehat{\nabla}^{\mathbb{R}^{(1,2)}}}(\partial_t,\partial_\eta)\partial_t=-\partial_\eta,\nonumber\\
&{R}^{\widehat{\nabla}^{\mathbb{R}^{(1,2)}}}(\partial_\xi,\partial_t)\partial_t=\partial_\xi,~~~
{R}^{\widehat{\nabla}^{\mathbb{R}^{(1,2)}}}(\partial_\eta,\partial_t)\partial_t=\partial_\eta,\nonumber\\
&{R}^{\widehat{\nabla}^{\mathbb{R}^{(1,2)}}}(\partial_{x^J},\partial_{x^K})\partial_{x^L}=0,\nonumber\\
\end{align}
for other pairs $(\partial_{x^J},\partial_{x^K},\partial_{x^L})$.
By (\ref{a4}) and (\ref{c9}), we have
\begin{align}\label{c10}
&{\rm Ric}^{\widehat{\nabla}^{\mathbb{R}^{(1,2)}}}(\partial_t,\partial_t)=2,~~~
{\rm Ric}^{\widehat{\nabla}^{\mathbb{R}^{(1,2)}}}(\partial_{x^J},\partial_{x^L})=0,
\end{align}
for other pairs $(\partial_{x^J},\partial_{x^L})$.\\
\indent If $(\widetilde{M},g_\mu,\widehat{\nabla}^{\mu})$ is Einstein with the Einstein constant
    $\lambda$, by Propsition \ref{prop1} and (\ref{c10}), we have
 \begin{equation}\label{c11}
 \lambda=0,~~2-(q-n)(\frac{h''}{h}-1)=-\lambda.
\end{equation}
 Solving (\ref{c11}), we get
 \begin{equation}\label{c12}
 h=c_1e^{{\sqrt{1+\frac{2}{q-n}}}t}+c_2e^{-{\sqrt{1+\frac{2}{q-n}}}t}.
\end{equation}
 By (\ref{b19}) (3) and the Einstein condition, we get
$(M_2^{(q,n)},\nabla^{L,M_2})$ is Einstein with the Einstein constant $c_0$ and
\begin{equation}\label{c13}
\lambda h^2-h''h-(q-n-1)(h')^2+(q-n-2)hh'=c_0.\\
\end{equation}
Then we have the following theorem
\begin{thm}\label{cthm6}
Let $\widetilde{M}=\mathbb{R}^{(1,2)}\times _\mu M_2^{(q,n)}$ and $g_\mu=g_1+h^2g_2$ and $P=\partial_t$. Then
    $(\widetilde{M},g_\mu,\widehat{\nabla}^{\mu})$ is Einstein with the Einstein constant
    $\lambda$ if and only if $( M_2^{(q,n)},\nabla^{L,M_2})$ is Einstein with the Einstein constant $c_0=0$ and $\lambda=0$, $h=c^*$, $q-n+2=0$.
\end{thm}
\begin{proof}
Let $k=1+\frac{2}{q-n}$, by $\lambda=0,$ (\ref{c12}) and (\ref{c13}), we have
\begin{align}\label{hhh}
&[-(q-n-1)k^2c_1^2+(q-n-2)kc_1^2]e^{2kt}+[-(q-n-1)k^2c_2^2+(q-n-2)kc_2^2]e^{-2kt}-c_1k^2e^{kt}-c_2k^2e^{-kt}\nonumber\\
&=c_0-2c_1c_2(q-n-1)k^2.\nonumber\\
\end{align}
Let $b_1=-(q-n-1)k^2c_1^2+(q-n-2)kc_1^2,$ $b_2=-(q-n-1)k^2c_2^2+(q-n-2)kc_2^2,$ $b_3=-c_1k^2,$ $b_4=-c_2k^2,$ $b_5=c_0-2c_1c_2(q-n-1)k^2,$ we get
\begin{align}\label{h1h}
&b_1e^{2kt}+b_2e^{-2kt}+b_3e^{kt}+b_4e^{-kt}=b_5.\nonumber\\
\end{align}
When $k\neq 0,$ we have
\begin{eqnarray}\label{eeeeee}
       \begin{cases}
       b_1+b_2+b_3+b_4=b_5 \\[2pt]
      2kb_1-2kb_2+kb_3-kb_4=0\\[2pt]
       4k^2b_1+4k^2b_2+k^2b_3+k^2b_4=0\\[2pt]
       8k^3b_1-8k^3b_2+k^3b_3-k^3b_4=0\\[2pt]
       16k^4b_1+416k^4b_2+k^4b_3+k^4b_4=0,\\[2pt]
       \end{cases}
\end{eqnarray}
by (\ref{eeeeee}), we get $b_1=b_2=b_3=b_4=b_5=0,$ then $c_1=c_2=0,$ so this is a contradiction.\\
When $k=0,$ we get $q-n+2=0,$ $c_0=0$ and $h=c_1+c_2=c^*.$\\
\end{proof}

\section*{Acknowledgements}
The second author was supported in part by  NSFC No.11771070. The authors thak the referee for his (or her) careful reading and helpful comments.

\clearpage
 \section*{Statement of `` Super warped products with a semi-symmetric non-metric connection"}
 a. Conflict of Interest Statement:
The authors declare no conflict of interest.\\

b. Data Availability Statement:
The authors confirm that the data supporting the findings of this study are available within the article.\\

c. Funding Information:
This research was funded by National Natural Science Foundation of China:
No.11771070.\\

d. Author Contribution Statement:
All authors contributed to the study conception and design. Material
preparation, data collection and analysis were performed by Tong Wu and Yong Wang. The first draft of
the manuscript was written by Tong Wu and all authors commented on previous versions of the manuscript.
All authors read and approved the final manuscript.\\

\end{document}